%% file: lanchier_reed_debts_2018.tex
\documentclass[letterpaper,11pt,twoside,keywordsasfootnote,addressatend,noinfoline]{article}
\usepackage{fullpage}
\usepackage[english]{babel}
\usepackage{amssymb}
\usepackage{amsmath}
\usepackage{theorem}
\usepackage{epsfig}
\usepackage{subfigure}
\usepackage{mathrsfs}
\usepackage{color}
\usepackage{imsart}

\newtheorem{theorem}{Theorem}

\newtheorem{lemma}[theorem]{Lemma}

\newenvironment{proof}{\noindent{\bf Proof.}}{\hspace*{2mm}~$\square$}
\newenvironment{proofof}[1]{\noindent{\bf Proof of #1.}}{\hspace*{2mm}~$\square$}

\newcommand{\N}{\mathbb{N}}
\newcommand{\Z}{\mathbb{Z}}
\newcommand{\R}{\mathbb{R}}

\newcommand{\G}{\mathscr{G}}
\newcommand{\V}{\mathscr{V}}
\newcommand{\E}{\mathscr{E}}
\newcommand{\C}{\mathscr{C}}
\newcommand{\D}{\mathscr{D}}

\newcommand{\ind}{\mathbf{1}}
\newcommand{\ep}{\epsilon}
\newcommand{\n}{\hspace*{-5pt}}
\newcommand{\il}{L_i}
\newcommand{\cl}{L_c}

\DeclareMathOperator{\card}{card}
\DeclareMathOperator{\uniform}{Uniform \,}

\DeclareMathOperator{\sign}{sign \,}

\begin{document}

\begin{frontmatter}
\title     {Rigorous results for the distribution of money \\ on connected graphs (models with debts)}
\runtitle  {Distribution of money for models with debts}
\author    {Nicolas Lanchier and Stephanie Reed}
\runauthor {Nicolas Lanchier and Stephanie Reed}
\address   {School of Mathematical and Statistical Sciences \\ Arizona State University \\ Tempe, AZ 85287, USA. \\ nicolas.lanchier@asu.edu \\ sjreed5@asu.edu}

\maketitle

\begin{abstract} \ \
 In this paper, we continue our analysis of spatial versions of agent-based models for the dynamics of money that have been
 introduced in the statistical physics literature, focusing on two models with debts.
 Both models consist of systems of economical agents located on a finite connected graph representing a social network.
 Each agent is characterized by the number of coins she has, which can be negative in case she is in debt, and each monetary transaction consists
 in one coin moving from one agent to one of her neighbors.
 In the first model, that we name the model with individual debt limit, the agents are allowed to individually borrow up to a fixed number of coins.
 In the second model, that we name the model with collective debt limit, agents can borrow coins from a central bank as long as the bank is not empty,
 with reimbursements occurring each time an agent in debt receives a coin.
 Based on numerical simulations of the models on complete graphs, it was conjectured that, in the large population/temperature limits, the
 distribution of money converges to a shifted exponential distribution for the model with individual debt limit, and to an asymmetric Laplace distribution
 for the model with collective debt limit.
 In this paper, we prove exact formulas for the distribution of money that are valid for all possible social networks.
 Taking the large population/temperature limits in the formula found for the model with individual debt limit, we prove convergence to the shifted
 exponential distribution, thus establishing the first conjecture.
 Simplifying the formula found for the model with collective debt limit is more complicated, but using a computer to plot this formula
 shows an almost perfect fit with the Laplace distribution, which strongly supports the second conjecture.
\end{abstract}



\end{frontmatter}



\section{Introduction}
\label{paper3:intro}
 The main objective of this paper is to continue the mathematical analysis of economical models for the dynamics of money initiated by
 the authors in~\cite{lanchier_2017b, lanchier_reed_2018}.
 These models consist of (typically) large systems of economical agents, where each agent is characterized by the amount of money, or number of
 coins, she has at a given time.
 The processes evolve in discrete time and, at each time step, two agents are selected at random from the entire population to engage in a monetary
 transaction.
 The main problem about these models is to find the limiting distribution of money, i.e., letting~$\xi_t (x)$ be the number of coins agent~$x$ has
 at time~$t$, the objective is to find
 $$ \lim_{t \to \infty} P (\xi_t (x) = c) $$
 the probability that agent~$x$ has~$c$ coins at equilibrium.
 Physicists have introduced a number of economical models and, relying on numerical simulations, were able to derive interesting conjectures about
 the distribution of money in the limit as the number of individuals, and the average number of coins per individual called the money temperature,
 both tend to infinity.
 The simplest such system, which we call the {\bf one-coin model}, was introduced in~\cite{dragulescu_yakovenko_2000}.
 In this model, after two agents have been selected to engage in a monetary transaction, one of the two agents chosen at random gives one of her
 coins to the other agent (if she has at least one).
 The authors of~\cite{dragulescu_yakovenko_2000} conjectured that, for this simple model, the distribution of money converges to the exponential
 distribution with mean the money temperature in the large population/temperature limits. \\
\indent More realistic variants of the one-coin model where monetary transactions involve multiple coins were also introduced in the statistical
 physics literature, such as the three models studied analytically by the authors in~\cite{lanchier_reed_2018} that can be described as follows.
\begin{itemize}
 \item In the {\bf uniform reshuffling model} also introduced in~\cite{dragulescu_yakovenko_2000}, all the coins of the two agents
       selected to interact are uniformly redistributed between the two agents.
       The simulations performed in~\cite{dragulescu_yakovenko_2000} suggest that, for this type of monetary transaction, the distribution
       of money at equilibrium converges to the exponential distribution with mean the money temperature in the large population/temperature limits,
       just like in the one-coin model. \vspace*{4pt}
 \item In the {\bf immediate exchange model} introduced in~\cite{heinsalu_patriarca_2014}, each of the two interacting agents chooses independently
       and uniformly at random a number of her coins that she gives to the other interacting agent.
       Results from~\cite{heinsalu_patriarca_2014, katriel_2015} suggest that, in this case, the distribution of money converges to a gamma
       distribution with mean the money temperature and shape parameter two in the large population/temperature limits. \vspace*{4pt}
 \item In the {\bf model with saving propensity} introduced in~\cite{chakraborti_chakrabarti_2000}, the two interacting agents independently save
       some of their coins, and the remaining coins are uniformly redistributed between the two agents, just like in the uniform reshuffling model.
       The computer simulations performed in~\cite{patriarca_chakraborti_kashi_2004} suggest that, like in the immediate exchange model, the limiting
       distribution of money converges to a gamma distribution with mean the money temperature and shape parameter two in the large
       population/temperature limits.
\end{itemize}
 All the conjectures above have been recently proved analytically by the authors in~\cite{lanchier_2017b} for the one-coin model,
 and in~\cite{lanchier_reed_2018} for the other three models.
 In addition, we were able to extend these results to more general models where the economical agents are located on a finite connected graph
 and can only interact with their neighbors.
 Thinking of the graph as a social network, this means that individuals can only exchange money with their friends or business partners, which
 results in a more realistic model.
 Interestingly, we were able to prove that the distribution of the number of coins a given agent has at equilibrium does not depend on the
 number of connections of this agent, in particular the distribution is the same for agents with a large number of connections and for agents
 with only one neighbor. \\
\indent In this paper, we continue our analysis, looking now at two models with debts introduced in the physics
 literature~\cite{dragulescu_yakovenko_2000, xi_ding_wang_2005}.
 The inclusion of debts is modeled by the fact that the agents can now have a negative number of coins.
 The evolution rules at each transaction are the same as in the idealized one-coin model, but the models also differ qualitatively
 in that~\cite{dragulescu_yakovenko_2000} assumes that the agents have the same individual debt limit whereas~\cite{xi_ding_wang_2005}
 assumes that there is a collective limit.
\begin{itemize}
 \item We call model with {\bf individual debt limit} the model with debts introduced in~\cite{dragulescu_yakovenko_2000}.
       In this model, agents are allowed to individually borrow~$\il$ coins, and the numerical simulations performed
       in~\cite{dragulescu_yakovenko_2000} suggest that the distribution of money at equilibrium now converges to a shifted exponential
       distribution in the large population/temperature limits. \vspace*{4pt}
 \item We call model with {\bf collective debt limit} the model in~\cite{xi_ding_wang_2005}.
       In this model, agents can borrow coins from a bank that starts with~$\cl$ coins as long as the bank is not empty, with
       reimbursements occurring each time an agent in debt receives a coin.
       It was conjectured in~\cite{xi_ding_wang_2005} that the distribution of money at equilibrium converges to an asymmetric Laplace
       distribution in the large population/temperature limits.
\end{itemize}
 As in~\cite{lanchier_2017b, lanchier_reed_2018}, we study spatial generalizations of these economical models where agents are located
 on a social network.
 Following an approach similar to~\cite{lanchier_2017b}, we give a complete proof of (and extend) the conjecture in~\cite{dragulescu_yakovenko_2000}
 about the model with individual debt limit.
 The model with collective debt limit is more challenging.
 Our main result gives an exact expression of the distribution of money at equilibrium for all possible number of individuals and coins, but
 we were not able to simplify this expression in the large population/temperature limits to prove the conjecture in~\cite{xi_ding_wang_2005}.
 However, using a computer to plot the exact distribution found analytically shows an almost perfect fit with the Laplace distribution
 found numerically in~\cite{xi_ding_wang_2005}, which strongly supports their conjecture. \\
\indent In the next two sections, we give a rigorous definition of the spatial versions of the two models with debts introduced
 in~\cite{dragulescu_yakovenko_2000, xi_ding_wang_2005}, and state our main results about the distribution of money at equilibrium.
 The other sections are devoted to proofs.


\section{Model description}
\label{paper3:models}
 In contrast with~\cite{dragulescu_yakovenko_2000, xi_ding_wang_2005} that rely on numerical simulations restricted to models where all pairs of
 individuals are equally likely to interact at each time step, our analysis is general enough to account for local interactions and network structure.
 This means that the individuals are more realistically located on the set of vertices of a general finite connected graph~$\G = (\V, \E)$ representing
 a social network.
 The structure of the network is incorporated in the dynamics by assuming that only individuals located on vertices connected by an edge of the graph,
 that can be thought of as friends or business partners, may interact to exchange money.
 Note that the models in~\cite{dragulescu_yakovenko_2000, xi_ding_wang_2005} can be viewed as the particular cases where~$\G$ is a complete graph. \\
\indent The two models with debts studied in this paper are spatially explicit variants of the basic one-coin model introduced
 in~\cite{dragulescu_yakovenko_2000}, and studied analytically on general connected graphs in~\cite{lanchier_2017b}.
 Having a general finite connected graph~$\G = (\V, \E)$, the (spatial) one-coin model is a discrete-time Markov chain in which the state at time~$t$
 is a configuration
 $$ \xi_t : \V \to \N \quad \hbox{where} \quad \xi_t (x) = \hbox{number of coins agent~$x$ has at time~$t$}. $$
 In order to describe the dynamics, because the flow of money at each transaction is oriented from one vertex to another, it is convenient to define
 the set of directed edges
 $$ \vec{\E} = \{(x, y), (y, x) : \{x, y \} \in \E \}. $$
 At each time step, a directed edge, say~$(x, y) \in \vec{\E}$, is chosen uniformly at random, which results in the transfer of one coin
 from vertex~$x$ to vertex~$y$ if and only if there is a least one coin at~$x$ before the interaction.
 The restriction on the transfer reflects the fact that individuals cannot have debts.
 In contrast, the models introduced in~\cite{dragulescu_yakovenko_2000, xi_ding_wang_2005} allow the individuals to go into debt and have a negative
 number of coins.
 In particular, the state at time~$t$ is now
 $$ \xi_t : \V \to \Z \quad \hbox{where} \quad \xi_t (x) = \left\{\begin{array}{l}
        + \ \hbox{number of coins~$x$ has when~$\xi_t (x) \geq 0$} \vspace*{2pt} \\
        - \ \hbox{number of coins~$x$ borrowed when~$\xi_t (x) < 0$}. \end{array} \right. $$
 As in the basic one-coin model, a directed edge~$(x, y) \in \vec{\E}$ representing the flow of money (one coin) is chosen uniformly at random at each time step.
 The conditions under which the transaction indeed occurs are however different. \vspace*{5pt} \\
\noindent {\bf Model with individual debt limit}.
 In the model with debts introduced in~\cite{dragulescu_yakovenko_2000}, agents are allowed to borrow individually up to~$\il$ coins,
 which is modeled by assuming that the transaction indeed occurs if and only if the state at vertex~$x$ exceeds~$-\il$.
 In equations, letting~$(x, y)$ be the selected edge and~$\xi$ the configuration before the interaction, the configuration after the
 interaction is
 $$ \begin{array}{rcl}
    (\sigma_{x,y} \,\xi)(z) & \n = \n & \xi (z) + \ind \{\xi (x) > -\il \ \hbox{and} \ z = y \} \vspace*{4pt} \\ && \hspace*{45pt}
                                          - \ind \{\xi (x) > -\il \ \hbox{and} \ z = x \} \quad \hbox{for all} \quad z \in \V, \end{array} $$
 obtained by moving one coin from~$x$ to~$y$ if and only if the state at~$x$ exceeds~$-\il$.
 Recalling that, at each time step, a directed edge is chosen uniformly at random, and using that the total number of directed edges is twice
 the number of edges in~$\E$, the model with individual debt limit is the discrete-time Markov chain~$(X_t)$ with transition probabilities
 $$ P (X_{t + 1} = \sigma_{x,y} \,\xi \,| \,X_t = \xi) = \frac{1}{2 \card (\E)} \quad \hbox{for all} \quad (x, y) \in \vec{\E}. $$
\noindent {\bf Model with collective debt limit}.
 The model with collective debt limit introduced in~\cite{xi_ding_wang_2005} is more complicated.
 The money can be borrowed from a central bank that we represent by adding a vertex~$\star$ to the vertex set~$\V$, so the vertex set becomes
 $$ \V^{\star} = \V \cup \{\star \} \quad \hbox{where} \quad \star = \hbox{location of the central bank}. $$
 Though there is some flow of money between the bank and the individuals, the edge set is unchanged, i.e., there is no edge between the
 central bank and the individuals.
 As previously, a directed edge, say~$(x, y)$, is chosen at each time step.
 From the point of view of~$x$,
\begin{itemize}
 \item In case~$x$ has at least one coin, one coin moves from vertex~$x$ to vertex~$y$, which results in the state at vertex~$x$ to decrease by one. \vspace*{2pt}
 \item In case~$x$ has zero coin or is in debt, and there is at least one coin in the bank, $x$ takes one coin from the bank to give to~$y$, which
       results in the state at~$x$ to decrease by one. \vspace*{2pt}
 \item In case~$x$ has zero coin or is in debt, and the bank has no coin, nothing happens.
\end{itemize}
 From the point of view of~$y$,
\begin{itemize}
 \item In case~$y$ is not in debt before the interaction and indeed receives one coin, the state of the bank does not change further and the state
       at vertex~$y$ increases by one. \vspace*{2pt}
 \item In case~$y$ is in debt before the interaction and indeed receives one coin, $y$ gives one coin to the bank to reimburse part of her
       debt, so the state of the bank increases by one and the state at vertex~$y$ increases by one.
\end{itemize}
 Note in particular that, in case both~$x$ and~$y$ are in debt and there is at least one coin at the bank before the interaction, the bank gives one
 coin to~$x$, and~$y$ immediately gives it back to the bank so the state of the bank does not change.
 In equations, letting~$\xi$ be the configuration before the interaction, the configuration after the interaction is
 $$ \begin{array}{rcl}
    (\tau_{x,y} \,\xi)(z) & \n = \n & \xi (z) + \ind \{\max (\xi (x), \xi (\star)) > 0 \ \hbox{and} \ z = y \} \vspace*{4pt} \\ && \hspace*{45pt}
                                          - \ \ind \{\max (\xi (x), \xi (\star)) > 0 \ \hbox{and} \ z = x \} \quad \hbox{for all} \quad z \in \V \vspace*{8pt} \\
    (\tau_{x,y} \,\xi)(\star) & \n = \n & \xi (\star) + \ind \{\max (\xi (x), \xi (\star)) > 0 \ \hbox{and} \ \xi (x) > 0 \ \hbox{and} \ \xi (y) < 0 \} \vspace*{4pt} \\ && \hspace*{45pt}
                                          - \ \ind \{\max (\xi (x), \xi (\star)) > 0 \ \hbox{and} \ \xi (x) \leq 0 \ \hbox{and} \ \xi (y) \geq 0 \}. \end{array} $$
 In particular, using again that there are~$2 \card (\E)$ directed edges, the model with collective debt limit is the Markov chain~$(Y_t)$ with
 transition probabilities
 $$ P (Y_{t + 1} = \tau_{x,y} \,\xi \,| \,Y_t = \xi) = \frac{1}{2 \card (\E)} \quad \hbox{for all} \quad (x, y) \in \vec{\E}. $$
 From now on, we let~$N = \card (\V)$ be the number of vertices, which is also the number of individuals in the system.
 Note that, each time a transaction indeed occurs, the state of one vertex decreases by one, the state of another vertex increases by one, and the state
 of all the other vertices does not change.
 In particular, letting~$M$ be the initial number of coins in the population,
 $$ \sum_{z \in \V} \,X_t (z) = \sum_{z \in \V} \,Y_t (z) = M \quad \hbox{for all} \quad t > 0. $$
 The model with individual debt limit is thus characterized by the structure of the network, and the two parameters~$M$ and~$\il$.
 Similarly, letting~$\cl$ be the initial number of coins in the central bank, the model with collective debt limit is characterized by the structure
 of the network, and the two parameters~$M$ and~$\cl$.
 Finally, the average number of coins per individual is denoted by~$T = M/N$, and called the money temperature by analogy with the notion of
 temperature in physics.


\section{Main results}
\label{paper3:results}
 As previously mentioned, the numerical simulations of the model with individual debt limit on the complete graph
 performed in~\cite{dragulescu_yakovenko_2000}, i.e., the model in which all pairs of individuals are equally likely to be selected, suggest
 that the limiting distribution of money approaches a shifted exponential distribution in the large population/temperature limits.
 The gray histogram in Figure~\ref{fig:individual} shows the distribution of money obtained from numerical simulations that we have reproduced.
 More precisely, the conjecture in~\cite{dragulescu_yakovenko_2000} states that, when the number of individuals~$N$ and the money
 temperature~$T = M/N$ are large, we have the approximation
\begin{equation}
\label{eq:dist-individual}
  \lim_{t \to \infty} P (X_t (x) = c) \approx f_X (c) = \frac{1}{T + \il} \ \exp \left(- \ \frac{c + \il}{T + \il} \right).
\end{equation}
 The black curve in Figure~\ref{fig:individual} shows the graph of the density function~$f_X$.
 Note that this distribution is similar to the distribution of money in the one-coin model except that the money temperature~$T$ is
 replaced by~$T + \il$, and the range is shifted by~$-\il$.
 This shows that the model with individual debt limit behaves essentially like the one-coin model in which the average number of coins
 per individual is now~$T + \il$, due to the fact that each individual can use~$\il$ additional coins.
 Using reversibility of the stochastic process to identify the stationary distribution and some basic combinatorics to count the total
 number of admissible configurations, we get the following theorem.
\begin{figure}[t]
\centering
\scalebox{0.32}{\input{debt.pstex_t}}
\caption{\upshape{Distribution of money for the model with individual debt limit~$\il = 1000$ and temperature~$T = 500$.
 The solid curve corresponds to the graph of the density function~\eqref{eq:dist-individual} while the gray histogram gives the distribution of money obtained
 from~$5 \times 10^{10}$ iterations of the process on the complete graph with~$N = 10,000$ vertices.}}
\label{fig:individual}
\end{figure}
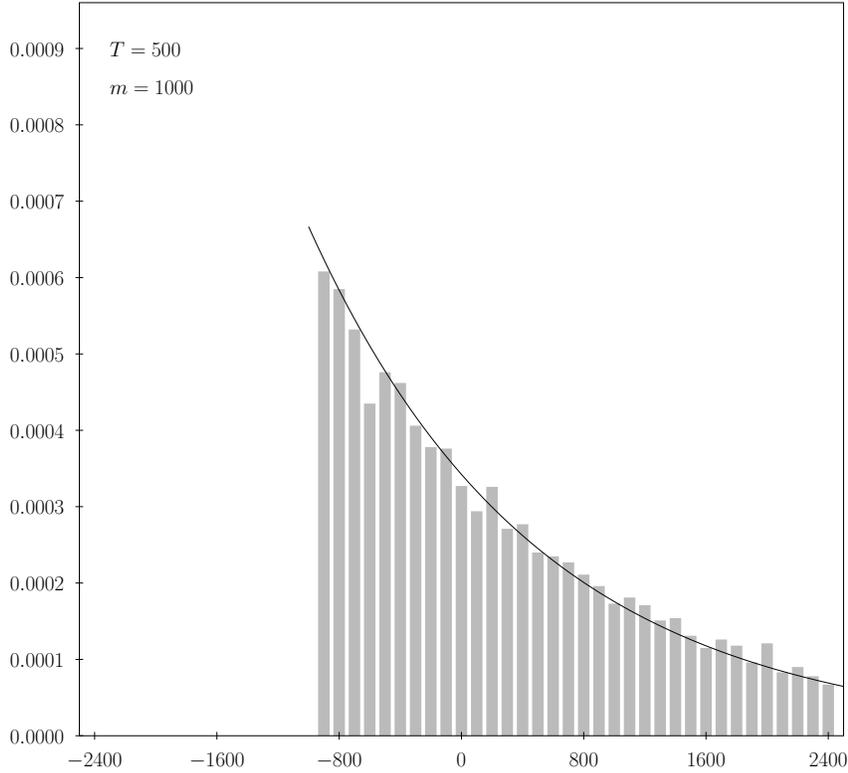
\begin{theorem}[individual debt limit] --
\label{th:individual}
 For all connected graph~$\G$ with~$N$ vertices, regardless of the number~$M$ of coins, the limit~$\il$ and the initial configuration,
 $$ \lim_{t \to \infty} P (X_t (x) = c) = \frac{\Lambda_X (N - 1, M - c, \il)}{\Lambda_X (N, M, \il)} \quad \hbox{for all} \quad -\il \leq c \leq M + \il (N - 1) $$
 where the function~$\Lambda_X$ is given by
 $$ \Lambda_X (N, M, \il) = {M + \il N + N - 1 \choose N - 1} \quad \hbox{for all} \quad N, M, \il. $$
 In particular, when~$N$ and~$T = M/N$ are large,
 $$ \lim_{t\to\infty}P (X_{t} (x) = c) \approx f_X (c) \quad \hbox{for all} \quad -\il \leq c \leq M + \il (N - 1). $$
\end{theorem}
 The first part of the theorem gives an exact expression of the distribution of money for all possible parameters of the system, and
 for all possible connected graphs.
 The second part shows that taking the limit as~$N$ and~$T$ tend to infinity in this exact expression implies the conjecture~\eqref{eq:dist-individual}. \\
\indent We now look at the model with collective debt limit.
 Recall that the numerical simulations of the model on the complete graph in~\cite{xi_ding_wang_2005} suggest that the distribution of money
 at equilibrium now approaches an asymmetric Laplace distribution in the large population/temperature limits.
 The gray histogram in Figure~\ref{fig:collective} shows the distribution of money obtained from numerical simulations that we have reproduced.
 More precisely, the conjecture in~\cite{xi_ding_wang_2005} states that, in the large population limit, and when~$M/N$ and~$\cl/N$ are large, we have
 the approximation
\begin{equation}
\label{eq:dist-collective-1}
  \lim_{t \to \infty} P (X_t (x) = c) \approx f_Y (c) = \left\{\hspace*{-3pt} \begin{array}{lcl} K \,e^{- ac} & \hbox{for} & c \geq 0 \vspace*{3pt} \\
                                                                                                 K \,e^{+ bc}  & \hbox{for} & c \leq 0 \end{array} \right.
\end{equation}
 where, letting~$\rho = \cl/M$,
\begin{equation}
\label{eq:dist-collective-2}
   K \sim \frac{1}{T} \bigg(\sqrt{1 + \rho} - \sqrt{\rho} \bigg)^2 \quad
   a \sim \frac{1}{T} \bigg(1 - \sqrt{\frac{\rho}{1 + \rho}} \bigg) \quad
   b \sim \frac{1}{T} \bigg(\sqrt{\frac{1 + \rho}{\rho}} - 1 \bigg).
\end{equation}
 The black curve in Figure~\ref{fig:collective} shows the graph of the density function~$f_Y$.
 As for the model with individual debt limit, using reversibility to identify the stationary distribution and some combinatorics to count the total
 number of admissible configurations leads to an exact expression of the distribution of money for all possible connected graphs and all
 possible~$M$ and~$\cl$.
 The combinatorial analysis, however, is more complicated for the model with collective debt limit.
\begin{theorem}[collective debt limit] --
\label{th:collective}
 For all connected graph~$\G$ with~$N$ vertices, regardless of the number~$M$ of coins, the limit~$\cl$ and the initial configuration,
 $$ \lim_{t \to \infty} P (Y_t (x) = c) = \left\{\begin{array}{ccl}
    \displaystyle \frac{\Lambda_Y (N - 1, M - c, \cl)}{\Lambda_Y (N, M, \cl)} & \hbox{for all} & 0 \leq c \leq M + \cl \vspace*{8pt} \\
    \displaystyle \frac{\Lambda_Y (N - 1, M - c, \cl + c)}{\Lambda_Y (N, M, \cl)} & \hbox{for all} & - \cl \leq c \leq 0 \end{array} \right. $$
 where the function~$\Lambda_Y$ is given by
 $$ \Lambda_Y (N, M, \cl) = \sum_{a = 0}^{\cl} \ \sum_{b = 0}^N {N \choose b} {a - 1 \choose b - 1} {M + a + N - b -1 \choose N - b - 1}. $$
\end{theorem}
 In contrast with Theorem~\ref{th:individual}, we were not able to simplify the expression of~$\Lambda_Y$ sufficiently to prove that the distribution
 of money indeed converges to~\eqref{eq:dist-collective-1}--\eqref{eq:dist-collective-2} in the large population/temperature limits.
 However, using a computer program to plot the exact expression of the distribution of money found in the theorem (the black squares in Figure~\ref{fig:collective})
 shows an almost perfect fit with the~Laplace distribution (the solid curve in Figure~\ref{fig:collective}) found via numerical simulations
 in~\cite{xi_ding_wang_2005}, which strongly supports their conjecture.
 To further support the conjecture, we also give a partial proof of~\eqref{eq:dist-collective-2}, assuming that~\eqref{eq:dist-collective-1} holds,
 in the last section of this paper.
 More precisely, we prove that, at equilibrium, the number of coins in the bank is a strict supermartingale, which suggests that the number of
 coins in the bank divided by the collective debt limit~$\cl$ converges to zero as~$\cl \to \infty$.
 Assuming that this is indeed the case and that the distribution of money is indeed described by an asymmetric~Laplace distribution~\eqref{eq:dist-collective-1},
 we prove~\eqref{eq:dist-collective-2} rigorously. \\
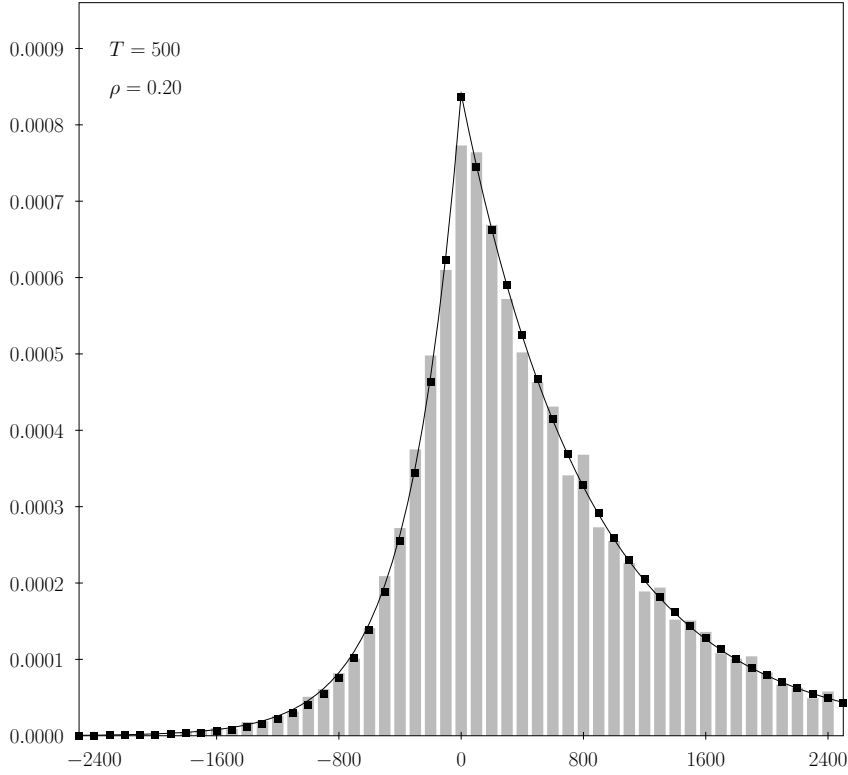
\begin{figure}[t]
\centering
\scalebox{0.32}{\input{bank.pstex_t}}
\caption{\upshape{Distribution of money for the model with collective debt limit with money temperature~$T = 500$ and parameter~$\rho = 0.20$.
 The solid curve corresponds to the graph of the density function~\eqref{eq:dist-collective-1}--\eqref{eq:dist-collective-2} while the black squares
 correspond to the distribution of money obtained from the exact expression in Theorem~\ref{th:collective} with~$N = 100$.
 The gray histogram gives the distribution of money obtained from~$5 \times 10^{10}$ iterations of the stochastic process on the complete graph with~$N = 10,000$
 vertices.}}
\label{fig:collective}
\end{figure}
\indent The rest of the paper is devoted to proofs and organized as follows.
 In the next section, we collect preliminary results for both models, showing the existence and uniqueness of their stationary distribution.
 We also use reversibility to prove that, at equilibrium, all the configurations are equally likely.
 In the following two sections, we use some combinatorics to obtain an explicit expression of the stationary distribution, from which
 Theorems~\ref{th:individual} and~\ref{th:collective} can be easily deduced.
 Finally, the last section gives the partial proof of~\eqref{eq:dist-collective-2} under the assumption that~\eqref{eq:dist-collective-1} holds.


\section{Ergodicity and reversibility}
\label{paper3:reversibility}
 This section collects some preliminary results about the models with individual debt limit and collective debt limit that will be useful later to prove our theorems.
 More precisely, we prove that both processes are irreducible, aperiodic and reversible.
 From now on, the set of all possible configurations of the model with individual debt limit is denoted by
 $$ \begin{array}{l} \C_{N, M, \il} = \{\xi : \V \to \Z : \sum_{x \in \V} \xi (x) = M \ \hbox{and} \ \xi (x) \geq -\il \ \hbox{for all} \ x \in \V \} \end{array} $$
 while the set of configurations of the model with collective debt limit is
 $$ \begin{array}{l} \D_{N, M, \cl} = \{\xi : \V \to \Z : \sum_{x \in \V} \xi (x) = M \ \hbox{and} \ \sum_{x \in \V} (- \xi (x)) \,\ind \{\xi (x) < 0 \} \leq \cl \}. \end{array} $$
 Because the proofs of irreducibility, aperiodicity, and reversibility are similar for both models, instead of studying both processes separately,
 we give the details of the proofs for the process with individual debt limit and then briefly explain how the proof can be adapted to show the analog
 for the process with collective debt limit.
 We start with irreducibility.
 Intuitively, the reason why the two processes are irreducible is because the graph is connected, which allows us to move coins from any vertex to any other vertices.
\begin{lemma}[irreducibility] --
\label{lem:X-irreducible}
 The process~$(X_t)$ is irreducible.
\end{lemma}
\begin{proof}
 Letting~$\xi \neq \xi'$, the objective is to show that
\begin{equation}
\label{eq:X-irreducible-0}
  P (X_t = \xi' \,| \,X_0 = \xi) > 0 \quad \hbox{for some} \quad t \in \N.
\end{equation}
 The key is to use the following metric~$d$ on the set of configurations:
 $$ d (\xi,\xi') = \sum_{z \in \V} \,|\xi (z) - \xi' (z)| \quad \hbox{for all} \quad \xi, \xi' : \V \to \Z. $$
 Because~$\xi \neq \xi'$, there exist~$x, y \in \V$ such that
 $$ \xi (x) > \xi' (x) \quad \hbox{and} \quad \xi (y) < \xi' (y). $$
 Then, define the configuration~$\eta \in \C_{N, M, \il}$ as
 $$ \eta (x) = \xi (x) - 1, \quad \eta (y) = \xi (y) + 1, \quad \eta (z) = \xi (z) \ \hbox{for all} \ z \neq x, y $$
 and let~$\Gamma = (x, z_1, z_2, \ldots, z_s, y)$ be a self-avoiding path of the graph~$\G$ connecting~$x$ and~$y$.
 Note that such a path exists because the graph is connected.
 Note also that, because~$\Gamma$ is a self-avoiding path, the integer~$s$ is less than the diameter of the graph.
 Now, define recursively
 $$ \xi_1 = \tau_{x, z_1} \,\xi, \ \xi_2 = \tau_{z_1, z_2} \,\xi_1, \ldots, \ \xi_{i + 1} = \tau_{z_i, z_{i + 1}} \,\xi_i, \ldots, \ \xi_{s + 1} = \tau_{z_s, y} \,\xi_s. $$
 Because~$\xi (x) > \xi' (x) \geq -\il$, we can move a coin from~$x$ to~$z_1$, which gives~$\xi_1 (z_1) > -\il$.
 By using a simple induction, we deduce that~$\xi_i (z_i) > -\il$ for all~$i$.
 This also implies that we can move a coin~$s + 1$ times to bring it from~$x$ to~$y$ therefore~$\xi_{s + 1} = \eta$ and
\begin{equation}
\label{eq:X-irreducible-1}
  P (X_{s + 1} = \eta \,| \,X_0 = \xi) \geq P_{\xi} (X_1 = \xi_1) \ \prod_{i = 1}^s \,P_{\xi_i} (X_1 = \xi_{i + 1}) = \bigg(\frac{1}{2 \card (\E)} \bigg)^{s + 1}.
\end{equation}
 In addition, because~$\xi (x) > \xi' (x)$, we must have
\begin{equation}
\label{eq:X-irreducible-2}
  \begin{array}{rcl}
   |\eta (x) - \xi' (x)| & \n = \n & \eta (x) - \xi' (x) = (\xi (x) - 1) - \xi' (x) \vspace*{4pt} \\
                         & \n = \n & (\xi (x) - \xi' (x)) - 1 = |\xi (x) - \xi' (x)| - 1. \end{array}
\end{equation}
 Similarly, because~$\xi (y) < \xi' (y)$, we have
\begin{equation}
\label{eq:X-irreducible-3}
  \begin{array}{rcl}
   |\eta (y) - \xi' (y)| & \n = \n & \xi' (y) - \eta (y) = \xi' (y) - (\xi (y) + 1) \vspace*{4pt} \\
                         & \n = \n & (\xi' (y) - \xi (y)) - 1 = |\xi (y) - \xi' (y)| - 1. \end{array}
\end{equation}
 Combining~\eqref{eq:X-irreducible-2} and~\eqref{eq:X-irreducible-3}, we deduce that
\begin{equation}
\label{eq:X-irreducible-4}
  \begin{array}{rcl}
    d (\eta, \xi') & \n = \n & \displaystyle \sum_{z \in \V} \,|\eta (z) - \xi' (z)| \vspace*{4pt} \\
                   & \n = \n & \displaystyle \sum_{z \neq x, y} |\eta (z) - \xi' (z)| + |\eta (x) - \xi' (x)| + |\eta (y) - \xi' (y)| \vspace*{4pt} \\
                   & \n = \n & \displaystyle \sum_{z \neq x, y} |\eta (z) - \xi' (z)| + (|\xi (x) - \xi' (x)| - 1) + (|\xi (y) - \xi' (y)| - 1) \vspace*{4pt} \\
                   & \n = \n & \displaystyle \sum_{z \in \V} \,|\xi (z) - \xi' (z)| - 2 = d (\xi, \xi') - 2 < d (\xi, \xi'). \end{array}
\end{equation}
 Observing also that~$-\il \leq \xi (z) \leq M + \il (N - 1)$ for all~$z \in \V$, we have
\begin{equation}
\label{eq:X-irreducible-5}
  d (\xi, \xi') = \sum_{z \in \V} \,|\xi (z) - \xi' (z)| \leq \sum_{z \in \V} \,(M + N\il) \leq N (M + N\il)
\end{equation}
 for all~$\xi, \xi' \in \C_{N, M, \il}$.
 Combining~\eqref{eq:X-irreducible-1}, \eqref{eq:X-irreducible-4} and~\eqref{eq:X-irreducible-5}, we conclude that
 $$ P (X_t = \xi' \,| \,X_0 = \xi) \geq \bigg(\frac{1}{2 \card (\E)} \bigg)^{DN (M + N\il)} $$
 for some~$t \leq DN (M + N\il)$, where~$D$ is the diameter of the graph~$\G$.
 This shows that~\eqref{eq:X-irreducible-0} holds so the process is irreducible and the proof is done.
\end{proof}
\begin{lemma}[irreducibility] --
\label{lem:Y-irreducible}
 The process~$(Y_t)$ is irreducible.
\end{lemma}
\begin{proof}
 The proof is identical to the proof of Lemma~\ref{lem:X-irreducible} with only exception:
 $$ -\cl \leq \xi (z) \leq M + \cl \quad \hbox{for all} \quad \xi \in \D_{N, M, \cl} \quad \hbox{and} \quad z \in \V. $$
 This implies that, for all~$\xi, \xi' \in \D_{N, M, \cl}$,
 $$ d (\xi, \xi') = \sum_{z \in \V} \,|\xi (z) - \xi' (z)| \leq \sum_{z \in \V} \,(M + 2\cl) \leq N (M + 2\cl) $$
 so the same argument as in the proof of Lemma~\ref{lem:X-irreducible} gives
 $$ P (X_t = \xi' \,| \,X_0 = \xi) \geq \bigg(\frac{1}{2 \card (\E)} \bigg)^{DN (M + 2\cl)} $$
 for some~$t \leq DN (M + 2\cl)$, where~$D$ is the diameter of the graph~$\G$.
\end{proof} \\ \\
 In view of irreducibility, in order to prove that the two processes are aperiodic, it suffices to identify a configuration that has period one,
 which is done in the next two lemmas.
\begin{lemma}[aperiodicity] --
\label{lem:X-aperiodic}
 The process~$(X_t)$ is aperiodic.
\end{lemma}
\begin{proof}
 Let~$\xi \in \C_{N, M, \il}$ such that~$\xi (x) = -\il$ for some~$x \in \V$.
 Because the graph~$\G$ is connected, vertex~$x$ has at least one neighbor~$y \in \V$.
 Also, given that the directed edge~$(x, y)$ is the one chosen at the next time step, because vertex~$x$ has reached its debt limit, it cannot
 give any coin to~$y$ so the exchange of money is canceled.
 In particular,
 $$ P (X_{t + 1} = \xi \,| \,X_t = \xi) \geq P (\hbox{edge~$(x, y) \in \vec{\E}$ is selected}) = \frac{1}{2 \card (\E)} > 0. $$
 This shows that configuration~$\xi$ has period one.
 Because the process is irreducible, all the configurations must have the same period, therefore the process is aperiodic.
\end{proof}
\begin{lemma}[aperiodicity] --
\label{lem:Y-aperiodic}
 The process~$(Y_t)$ is aperiodic.
\end{lemma}
\begin{proof}
 Starting with a configuration~$\xi \in \D_{N, M, \cl}$ such that~$\xi (x) = -\cl$ for some~$x \in \V$, and following the exact same reasoning as
 in the proof of Lemma~\ref{lem:X-aperiodic}, give the result.
\end{proof} \\ \\
 Irreducibility and aperiodicity, together with the fact that the number of configurations is finite, imply that, for both models, there
 exists a unique stationary distribution to which the process converges starting from any initial configuration.
 The next natural step is to find this stationary distribution for each process.
 In view of the large number of configurations and more importantly the fact that the agents are located on a general finite connected graph rather
 than a complete graph, writing the transition matrix in order to compute the stationary distribution looks impossible.
 However, one can easily find the stationary distribution observing that the processes are reversible and using the corresponding detailed
 balance equations.
\begin{lemma}[reversibility] --
\label{lem:X-reversible}
 The process~$(X_t)$ is reversible and
 $$ P (X_{t + 1} = \xi' \,| \,X_t = \xi) = P (X_{t + 1} = \xi \,| \,X_t = \xi') \quad \hbox{for all} \quad \xi, \xi' \in \C_{N, M, \il}. $$
\end{lemma}
\begin{proof}
 The equations to be proved are obvious when~$\xi = \xi'$.
 To deal with the nontrivial case where the two configurations are different, we distinguish two scenarios:
 $$ \begin{array}{rl}
    \hbox{(a)} & P (X_{t + 1} = \xi' \,| \,X_t = \xi) > 0 \ \ \hbox{and} \ \ \xi \neq \xi' \vspace*{4pt} \\
    \hbox{(b)} & P (X_{t + 1} = \xi' \,| \,X_t = \xi) = 0 \ \ \hbox{and} \ \ \xi \neq \xi'. \end{array} $$
 Because the configurations that can be reached from~$\xi$ in one step are of the form~$\sigma_{x, y} \,\xi$, in the context of scenario~(a),
 we have~$\xi (x) > -\il$ and
\begin{equation}
\label{eq:X-reversible-1}
  \xi' (x) = \xi (x) - 1, \quad \xi' (y) = \xi (y) + 1, \quad \xi' (z) = \xi (z) \ \hbox{for all} \ z \neq x, y
\end{equation}
 for some~$(x, y) \in \vec{\E}$.
 In particular,
 $$ \xi' (y) = \xi (y) + 1 \geq -\il + 1 > -\il $$
 which, together with~\eqref{eq:X-reversible-1}, implies that~$\sigma_{y, x} \,\xi' = \xi$. Therefore,
 $$ \begin{array}{rcl}
      P (X_{t + 1} = \xi' \,| \,X_t = \xi) & \n = \n & P (X_{t + 1} = \sigma_{x, y} \,\xi \,| \,X_t = \xi) = 1 / (2 \card (\E)) \vspace*{4pt} \\
                                           & \n = \n & P (X_{t + 1} = \sigma_{y, x} \,\xi' \,| \,X_t = \xi') = P (X_{t + 1} = \xi \,| \,X_t = \xi') \end{array} $$
 because the probabilities of choosing directed edge~$(x, y)$ and directed edge~$(y, x)$ (out of all the possible~$2 \card (\E)$ directed edges) are equal.
 Now, in the context of scenario~(b), condition~\eqref{eq:X-reversible-1} does not hold for any of the directed edges.
 Equivalently,
 $$ \xi (y) = \xi' (y) - 1, \quad \xi (x) = \xi' (x) + 1, \quad \xi (z) = \xi' (z) \ \hbox{for all} \ z \neq x, y $$
 does not hold for any of the directed edge~$(y, x) \in \vec{\E}$ from which it follows that configuration~$\xi$ cannot be reached from~$\xi'$ in one step.
 In conclusion, in the context of scenario~(b),
 $$ P (X_{t + 1} = \xi' \,| \,X_t = \xi) = P (X_{t + 1} = \xi \,| \,X_t = \xi') = 0. $$
 In either case, $P (X_{t + 1} = \xi' \,| \,X_t = \xi) = P (X_{t + 1} = \xi \,| \,X_t = \xi')$.
\end{proof}
\begin{lemma}[reversibility] --
\label{lem:Y-reversible}
 The process~$(Y_t)$ is reversible and
 $$ P (Y_{t + 1} = \xi' \,| \,Y_t = \xi) = P (Y_{t + 1} = \xi \,| \,Y_t = \xi') \quad \hbox{for all} \quad \xi, \xi' \in \D_{N, M, \cl}. $$
\end{lemma}
\begin{proof}
 As for irreducibility and aperiodicity, the proof of reversibility is quite similar for both processes, so we only focus on the differences.
 As previously, the detailed balanced equations are obvious when~$\xi = \xi'$ and the idea is again to distinguish between the two scenarios
 $$ \begin{array}{rl}
    \hbox{(a)} & P (Y_{t + 1} = \xi' \,| \,Y_t = \xi) > 0 \ \ \hbox{and} \ \ \xi \neq \xi' \vspace*{4pt} \\
    \hbox{(b)} & P (Y_{t + 1} = \xi' \,| \,Y_t = \xi) = 0 \ \ \hbox{and} \ \ \xi \neq \xi'. \end{array} $$
 In the context of scenario~(b), the same argument as in the proof of Lemma~\ref{lem:X-reversible} gives
 $$ P (Y_{t + 1} = \xi' \,| \,Y_t = \xi) = P (Y_{t + 1} = \xi \,| \,Y_t = \xi') = 0. $$
 In the context of scenario~(a), because the configurations that can be reached from~$\xi$ in one step are of the form~$\tau_{x, y} \,\xi$,
 we now have~$\max (\xi (x), \xi (\star)) > 0$ and
\begin{equation}
\label{eq:Y-reversible-1}
  \xi' (x) = \xi (x) - 1, \quad \xi' (y) = \xi (y) + 1, \quad \xi' (z) = \xi (z) \ \hbox{for all} \ z \neq x, y
\end{equation}
 for some~$(x, y) \in \vec{\E}$.
 Then, we have the following three implications:
 $$ \begin{array}{rcl}
    \xi (x) > 0 \ \ \hbox{and} \ \ \xi (\star) > 0    & \hbox{imply that} & \xi' (\star) \geq \xi (\star) > 0 \vspace*{4pt} \\
    \xi (x) \leq 0 \ \ \hbox{and} \ \ \xi (\star) > 0 & \hbox{imply that} & \xi' (y) > 0 \ \ \hbox{or} \ \ \xi' (\star) = \xi (\star) > 0 \vspace*{4pt} \\
    \xi (x) > 0 \ \ \hbox{and} \ \ \xi (\star) = 0    & \hbox{imply that} & \xi' (y) > 0 \ \ \hbox{or} \ \ \xi' (\star) = 1. \end{array} $$
 In either case, $\max (\xi' (y), \xi' (\star)) > 0$ which, together with~\eqref{eq:Y-reversible-1}, implies that~$\tau_{y, x} \,\xi' = \xi$.
 In particular, the two transition probabilities are equal:
 $$ \begin{array}{rcl}
      P (Y_{t + 1} = \xi' \,| \,Y_t = \xi) & \n = \n & P (Y_{t + 1} = \tau_{x, y} \,\xi \,| \,Y_t = \xi) = 1 / (2 \card (\E)) \vspace*{4pt} \\
                                           & \n = \n & P (Y_{t + 1} = \tau_{y, x} \,\xi' \,| \,Y_t = \xi') = P (Y_{t + 1} = \xi \,| \,Y_t = \xi'). \end{array} $$
 This completes the proof.
\end{proof}


\section{Proof of Theorem~\ref{th:individual}}
\label{paper3:individual}
 In this section, we prove the explicit expression for the distribution of money given in Theorem~\ref{th:individual} for the model with individual
 debt limit, and take the limit as~$N$ and~$T$ both tend to infinity to deduce and extend conjecture~\eqref{eq:dist-individual} to general finite
 connected graphs.
 As previously mentioned, irreducibility and aperiodicity of the model with individual debt limit proved in Lemmas~\ref{lem:X-irreducible}
 and~\ref{lem:X-aperiodic} and the fact that the number of configurations is finite imply that the process~$(X_t)$ has a unique stationary distribution,
 say~$\pi_X$, to which it converges starting from any initial configuration:
\begin{equation}
\label{eq:individual-1}
  \lim_{t \to \infty} P_{\eta} (X_t = \xi) = \pi_X (\xi) \quad \hbox{for all} \quad \xi, \eta \in \C_{N, M, \il}.
\end{equation}
 In addition, the detailed balanced equations in Lemma~\ref{lem:X-reversible} imply that, under the stationary distribution, all the configurations are
 equally likely.
 Indeed, the lemma shows that the transition matrix of the process is symmetric, and therefore doubly stochastic, from which it follows that
\begin{equation}
\label{eq:individual-2}
  \pi_X = \uniform (\C_{N, M, \il}).
\end{equation}
 Letting~$\Lambda_X (N, M, \il) = \card (\C_{N, M, \il})$ be the number of admissible configurations for the process with individual debt limit,
 it directly follows from~\eqref{eq:individual-1} and~\eqref{eq:individual-2} that, regardless of the initial configuration of the system and
 the choice of~$x \in \V$,
\begin{equation}
\label{eq:individual-3}
  \lim_{t \to \infty} P (X_{t} (x) = c) = \frac{\Lambda_X (N - 1, M - c, \il)}{\Lambda_X (N, M, \il)}
\end{equation}
 for all~$c = -\il, \ldots, M + \il (N - 1)$, because the numerator is equal to the number of configurations such that there are~$c$ coins at
 vertex~$x$, while the denominator is equal to the total number of configurations.
 In particular, to complete the proof of theorem, the next step is to count the number of configurations, which is done in the next lemma.
\begin{lemma} --
\label{lem:X-config}
 For all~$N, M$ and~$\il$, we have
 $$ \Lambda_X (N, M, \il) = {M + \il N + N - 1 \choose N - 1}. $$
\end{lemma}
\begin{proof}
 By definition, $\Lambda_X (N, M, \il)$ is the number of integer solutions to
 $$ \xi (x_1) + \xi (x_2) + \cdots + \xi (x_N) = M \quad \hbox{with} \quad \xi (x_i) \geq -\il \ \hbox{for all} \ i. $$
 Letting~$\eta (x_i) = \xi (x_i) + \il + 1$, this is the number of integer solutions to
 $$ \eta (x_1) + \eta (x_2) + \cdots + \eta (x_N) = M + N (\il + 1) \quad \hbox{with} \quad \eta (x_i) \geq 1 \ \hbox{for all} \ i $$
 which is known to be the binomial coefficient
 $$ {M + \il N + N - 1 \choose N - 1}. $$
 See for instance~\cite[Figure 1.3]{lanchier_2017a}.
\end{proof} \\ \\
 Using~\eqref{eq:individual-3} and Lemma~\ref{lem:X-config}, we are now ready to complete the proof of Theorem~\ref{th:individual}, which not only
 shows conjecture~\eqref{eq:dist-individual} but also extends the conjecture to all finite connected graphs. \\ \\
\begin{proofof}{Theorem~\ref{th:individual}}
 According to Lemma~\ref{lem:X-config},
 $$ \Lambda_X (N, M, \il) = \frac{1}{(N - 1)!} \,\prod_{k = 1}^{N - 1} \,(M + \il N + k). $$
 Similarly, we have
 $$ \Lambda_X (N - 1, M - c, \il) = \frac{1}{(N - 2)!} \,\prod_{k = 1}^{N - 2} \,(M - c + \il N - \il + k). $$
 Taking the ratio, we get
 $$ \begin{array}{rcl}
    \displaystyle \frac{\Lambda_X (N - 1, M - c, \il)}{\Lambda_X (N, M, \il)} & \n = \n &
    \displaystyle \frac{(N - 1)!}{(M + \il N + N - 1)(N - 2)!} \,\prod_{k = 1}^{N - 2} \,\bigg(\frac{M - c + \il N - \il + k}{M + \il N + k} \bigg) \vspace*{4pt} \\ & \n = \n &
    \displaystyle \frac{N - 1}{M + \il N + N - 1} \,\prod_{k = 1}^{N - 2} \,\bigg(\frac{M + \il N + k - (c + \il)}{M + \il N + k} \bigg) \vspace*{4pt} \\ & \n = \n &
    \displaystyle \frac{N - 1}{M + \il N + N - 1} \,\prod_{k = 1}^{N - 2} \,\bigg(1 - \frac{c + \il}{M + \il N + k} \bigg). \end{array} $$
 Recalling that~$T = M/N$, in the limit as~$N, T \to \infty$,
 $$ \begin{array}{rcl}
    \displaystyle \frac{\Lambda_X (N - 1, M - c, \il)}{\Lambda_X (N, M, \il)} & \n \sim \n &
    \displaystyle \bigg(\frac{N}{N (T + \il)} \bigg) \bigg(1 - \frac{c + \il}{N (T + \il)} \bigg)^N \vspace*{8pt} \\ & \n \sim \n &
    \displaystyle \bigg(\frac{1}{T + \il} \bigg) \exp \bigg(- \frac{c + \il}{T + \il} \bigg). \end{array} $$
 This, together with~\eqref{eq:individual-3}, completes the proof of the theorem.
\end{proofof}


\section{Proof of Theorem~\ref{th:collective}}
\label{paper3:collective}
 This section is devoted to the proof of Theorem~\ref{th:collective} that describes the limiting behavior of the model with collective debt limit.
 Following the same argument as in Section~\ref{paper3:individual}, but using
 Lemmas~\ref{lem:Y-irreducible}, \ref{lem:Y-aperiodic} and~\ref{lem:Y-reversible} instead of
 Lemmas~\ref{lem:X-irreducible}, \ref{lem:X-aperiodic} and~\ref{lem:X-reversible}, we obtain that
\begin{equation}
\label{eq:collective-1}
  \lim_{t \to \infty} P_{\eta} (Y_t = \xi) = \pi_Y (\xi) \quad \hbox{for all} \quad \xi, \eta \in \D_{N, M, \cl}
\end{equation}
 where the (unique) stationary distribution~$\pi_Y$ is
\begin{equation}
\label{eq:collective-2}
  \pi_Y = \uniform (\D_{N, M, \cl}).
\end{equation}
 Now, given that there are~$c$ coins at vertex~$x$, there is a total of~$M - c$ coins distributed across the rest of the graph.
 In addition, the number of coins from the bank all the agents excluding~$x$ can borrow is~$\cl$ when~$c \geq 0$ but~$\cl + c$ when~$c < 0$.
 Hence, letting~$\Lambda_Y (N, M, \cl) = \card (\D_{N, M, \cl})$ and applying~\eqref{eq:collective-1} and~\eqref{eq:collective-2}, we
 deduce that, regardless of the initial configuration,
 $$ \lim_{t \to \infty} P (Y_{t} (x) = c) = \left\{\begin{array}{ccl}
    \displaystyle \frac{\Lambda_Y (N - 1, M - c, \cl)}{\Lambda_Y (N, M, \cl)} & \hbox{for all} & 0 \leq c \leq M + \cl \vspace*{8pt} \\
    \displaystyle \frac{\Lambda_Y (N - 1, M - c, \cl + c)}{\Lambda_Y (N, M, \cl)} & \hbox{for all} & -\cl \leq c \leq 0. \end{array} \right. $$
 In particular, to complete the proof of Theorem~\ref{th:collective}, it suffices to compute the number of configurations~$\Lambda_Y (N, M, \cl)$,
 which is done in the next lemma.
\begin{lemma} --
\label{lem:Y-config}
 For all~$N, M$ and~$\cl$,
 $$ \Lambda_Y (N, M, \cl) = \sum_{a = 0}^{\cl} \ \sum_{b = 0}^N {N \choose b} {a - 1 \choose b - 1} {M + a + N - b -1 \choose N - b - 1}. $$
\end{lemma}
\begin{proof}
 Introducing the set
 $$ \D_{N, M, \cl}^+ = \{\xi \in \D_{N, M, \cl} : \xi (x) \geq 0 \ \hbox{for all} \ x \in \V \} $$
 and its complement~$\D_{N, M, \cl}^- = \D_{N, M, \cl} \setminus \D_{N, M, \cl}^+$, we have
\begin{equation}
\label{eq:Y-config-1}
  \Lambda_Y (N, M, \cl) = \card (\D_{N, M, \cl}) = \card (\D_{N, M, \cl}^+) + \card (\D_{N, M, \cl}^-).
\end{equation}
 Because~$\D_{N, M, \cl}^+$ is the set of configurations with nobody in debt, $\D_{N, M, \cl} = \C_{N, M, 0}$.
 In particular, a direct application of Lemma~\ref{lem:X-config} implies that
\begin{equation}
\label{eq:Y-config-2}
  \card (\D^+_{N, M, \cl}) = \card (\C_{N, M, 0}) = \Lambda_X (N, M, 0) = {M + N - 1 \choose N - 1}.
\end{equation}
 To count the number of configurations in~$\D_{N, M, \cl}^-$, we let
 $$ \begin{array}{rcl}
    \phi (N, M, \cl, a, b) & \n = \n &
    \hbox{number of configurations in~$\D_{N, M, \cl}$ with~$a$ coins} \vspace*{2pt} \\ & \n &
    \hbox{borrowed from the bank and where the debt} \vspace*{2pt} \\ & \n &
    \hbox{is shared among~$b$ individuals}. \end{array} $$
 Then, we have the following decomposition:
\begin{equation}
\label{eq:Y-config-3}
  \card (\D_{N, M, \cl}^-) = \sum_{a = 1}^{\cl} \sum_{b = 1}^{a \wedge (N - 1)} \phi (N, M, \cl, a, b).
\end{equation}
 Now, observe that
\begin{itemize}
 \item There are~$N$ choose~$b$ ways to choose the individuals in debt. \vspace*{4pt}
 \item Following the same reasoning as in the proof of Lemma~\ref{lem:X-config}, there are~$a - 1$ choose~$b - 1$ ways to distribute the debt
       among those individuals, \vspace*{4pt}
 \item Similarly, there are~$(M + a) + (N - b) - 1$ choose~$N - b - 1$ ways to distribute the~$M + a$ coins among the remaining~$N - b$ individuals.
\end{itemize}
 This implies that
\begin{equation}
\label{eq:Y-config-4}
  \phi (N, M, \cl, a, b) = {N \choose b} {a - 1 \choose b - 1} {M + a + N - b - 1 \choose N - b - 1}.
\end{equation}
 Finally, combining~\eqref{eq:Y-config-1}--\eqref{eq:Y-config-4} and using the convention
 $$ {-1 \choose -1} = 1 \quad \hbox{and} \quad {n \choose k} = 0 \quad \hbox{when} \quad n < k \ \ \hbox{or} \ \ k < 0 \leq n $$
 we conclude that
 $$ \begin{array}{rcl}
    \Lambda_Y (N, M, \cl) & \n = \n &
    \displaystyle {M + N - 1 \choose N - 1} + \sum_{a = 1}^{\cl} \,\sum_{b = 1}^{a \wedge (N - 1)} \phi (N, M, \cl, a, b) \vspace*{4pt} \\ & \n = \n &
    \displaystyle {M + N - 1 \choose N - 1} + \sum_{a = 1}^{\cl} \,\sum_{b = 1}^{a \wedge (N - 1)} {N \choose b} {a - 1 \choose b - 1} {M + a + N - b - 1 \choose N - b - 1} \vspace*{4pt} \\ & \n = \n &
    \displaystyle \sum_{a = 0}^{\cl} \ \sum_{b = 0}^N {N \choose b} {a - 1 \choose b - 1} {M + a + N - b -1 \choose N - b - 1}. \end{array} $$
 This completes the proof.
\end{proof}


\section{Partial proof of the conjecture~\eqref{eq:dist-collective-2}}
\label{paper3:heuristics}
 As previously mentioned, the plot (using a computer) of the explicit expression for the distribution of money proved in Theorem~\ref{th:collective}
 strongly suggests convergence to the asymmetric Laplace distribution conjectured in~\cite{xi_ding_wang_2005}.
 In this section, we give more evidence that this conjecture is true.
 More precisely, we assume convergence to an asymmetric Laplace distribution
\begin{equation}
\label{eq:heuristics-1}
  f_Y (c) = \left\{\hspace*{-3pt} \begin{array}{lcl} K \,e^{- ac} & \hbox{for} & c \geq 0 \vspace*{3pt} \\
                                                     K \,e^{+ bc}  & \hbox{for} & c \leq 0 \end{array} \right.
\end{equation}
 and argue that the three parameters~$K, a$ and~$b$ are indeed given by~\eqref{eq:dist-collective-2}.
 To compute these three parameters, the basic idea is to derive a system of three equations involving these parameters.
 The next lemma gives two such equations.
\begin{lemma} --
\label{lem:density-mean}
 For all~$N, M$ and~$\cl$, we have
 $$ \frac{K}{a} + \frac{K}{b} = 1 \quad \hbox{and} \quad \frac{K}{a^2} - \frac{K}{b^2} = \frac{M}{N}. $$
\end{lemma}
\begin{proof}
 The first equation directly follows from the fact that, because~$f_Y$ is a density function, its integral must be equal to one, while computing
 this integral gives
 $$ \int_{\R} f_Y (x) \,dx = \int_{\R_-} Ke^{+ bx} \,dx + \int_{\R_+} Ke^{- ax} \,dx = \frac{K}{a} + \frac{K}{b}. $$
 The second equation follows from looking at the mean number of coins per individual.
 Because each interaction has either no effect or moves one coin from a vertex~$x$ to a vertex~$y$, which results in
 the state at~$x$ to decrease by one and the state at~$y$ to decrease by one, we have
\begin{equation}
\label{eq:density-mean-1}
  \sum_{z \in \V} \,Y_t (z) = \sum_{z \in \V} \,Y_0 (z) = M \quad \hbox{for all} \quad t \in \R_+.
\end{equation}
 But our assumption that the distribution of money converges to the asymmetric Laplace distribution given in~\eqref{eq:heuristics-1} also implies that
\begin{equation}
\label{eq:density-mean-2}
  \begin{array}{rcl}
  \displaystyle \lim_{t \to \infty} E \bigg(\frac{1}{N} \sum_{z \in \V} Y_t (z) \bigg) & \n = \n &
  \displaystyle \lim_{t \to \infty} E (Y_t (x)) =
  \displaystyle \int_{\R} x f_Y (x) \,dx \vspace*{4pt} \\ & \n = \n &
  \displaystyle \int_{\R_-} Kxe^{+ bx} \,dx + \int_{\R_+} Kxe^{- ax} \,dx = \frac{K}{a^2} - \frac{K}{b^2}. \end{array}
\end{equation}
 Combining~\eqref{eq:density-mean-1} and~\eqref{eq:density-mean-2} gives the second equation.
\end{proof} \\ \\
 To find a third equation involving the parameters~$K, a$ and~$b$, the next step is to argue that the number of coins in the bank, when rescaled by its
 initial value~$\cl$, tends to zero as time goes to infinity.
 More precisely, we believe that
\begin{equation}
\label{eq:heuristics-3}
  \lim_{t \to \infty} \frac{Y_t (\star)}{\cl} \to 0 \quad \hbox{as} \quad N, M/N, \cl/N \to \infty.
\end{equation}
 The next lemma shows that, as long as there is at least one coin in the bank, the number of coins in the bank behaves like a supermartingale,
 which suggests that~\eqref{eq:heuristics-3} is true.
\begin{lemma} --
\label{lem:supermartingale}
 For all~$N, M, \cl$, we have
 $$ \lim_{t \to \infty} E (Y_{t + 1} (\star) - Y_t (\star) \,| \,Y_t (\star) > 0) < 0. $$
\end{lemma}
\begin{proof}
 Letting~$(x, y) \in \vec{\E}$ be the oriented edge selected at time~$t$, we define
 $$ J_t = (\sign (Y_t (x)), \sign (Y_t (y))) = \hbox{type of interaction at time~$t$}. $$
 In words, the random variable~$J_t$ keeps track of whether the agent~$x$ selected to give a coin and the agent~$y$ selected to receive
 a coin are in debt (state~$-$), have zero coin (state~0), or have a surplus of coins (state~$+$).
 There are~$3^2 = 9$ types of interactions and, as long as there is at least one coin in the bank, the type of interaction determines whether
 the number of coins in the bank decreases, stays the same, or increases. More precisely,
\begin{equation}
\label{eq:supermartingale-1}
  \begin{array}{rclcl}
    Y_{t + 1} (\star) & \n = \n & Y_t (\star) - 1 & \hbox{when} & J_t = (-, 0), (-, +), (0, 0), (0, +) \vspace*{4pt} \\
    Y_{t + 1} (\star) & \n = \n & Y_t (\star)     & \hbox{when} & J_t = (-, -), (0, -), (+, 0), (+, +) \vspace*{4pt} \\
    Y_{t + 1} (\star) & \n = \n & Y_t (\star) + 1 & \hbox{when} & J_t = (+, -). \end{array}
\end{equation}
 Now, define the corresponding conditional probabilities
 $$ p (\ep_1, \ep_2) = \lim_{t \to \infty} P (J_t = (\ep_1, \ep_2) \,| \,Y_t > 0) \quad \hbox{for all} \quad \ep_1, \ep_2 \in \{-, 0, + \}. $$
 Because edges~$(x, y)$ and~$(y, x)$ are equally likely to be chosen,
\begin{equation}
\label{eq:supermartingale-2}
  p (\ep_1, \ep_2) = p (\ep_2, \ep_1)  \quad \hbox{for all} \quad \ep_1, \ep_2 \in \{-, 0, + \}.
\end{equation}
 Also, using convergence to~$\pi_Y = \uniform (\D_{N, M, \cl})$ and that, regardless of the number of coins in the bank, there is
 a positive fraction of configurations with zero coin at a given vertex,
\begin{equation}
\label{eq:supermartingale-3}
  p (0, -) + p (0, 0) + p (0, +) = \lim_{t \to \infty} P (Y_t (x) = 0 \,| \,Y_t (\star) > 0) > 0.
\end{equation}
 Combining~\eqref{eq:supermartingale-1}--\eqref{eq:supermartingale-3}, we deduce that
 $$ \begin{array}{l}
    \lim_{t \to \infty} E (Y_{t + 1} (\star) - Y_t (\star) \,| \,Y_t (\star) > 0) \vspace*{4pt} \\ \hspace*{20pt} = \
      p (+, -) - p (-, 0) - p(-, +) - p (0, 0) - p (0, +) \vspace*{4pt} \\ \hspace*{20pt} = \
      - \ (p (-, 0) + p (0, 0) + p (0, +)) = - \lim_{t \to \infty} P (Y_t (x) = 0 \,| \,Y_t (\star) > 0) < 0. \end{array} $$
 This completes the proof.
\end{proof} \\ \\
 The previous lemma shows that the process~$(Y_t (\star))$ stopped at the time it reaches zero is a supermartingale.
 Because the inequality in the lemma is strict, this suggests that~\eqref{eq:heuristics-3} holds.
 To make the proof perfectly rigorous, we would need to prove a slightly stronger result, namely that the conditional expectation is less
 than a negative constant that does not depend on the parameters of the system.
 Numerical simulations of the model with collective debt limit support this result, revealing that the number of coins in the bank drops
 quickly and then fluctuates around values that are negligible compared to~$\cl$.
 Moving forward with our heuristic argument, we now assume that~\eqref{eq:heuristics-3} holds in order to derive a third equation involving
 the parameters~$K, a$ and~$b$.
\begin{lemma} --
\label{lem:third-equation}
 Assume that~\eqref{eq:heuristics-3} holds. Then,
 $$ \frac{K}{a^2} \sim \frac{M + \cl}{N} \quad \hbox{as} \quad N, M/N, \cl/N \to \infty. $$
\end{lemma}
\begin{proof}
 The trick to establish the lemma is now to look at the mean number of coins among the individuals who have no debt.
 Due to~\eqref{eq:heuristics-3}, the set of all individuals with no debt share a total of about~$M + \cl$ coins at equilibrium.
 In equation, this can be written as
\begin{equation}
\label{eq:third-equation-1}
  \lim_{t \to \infty} \ \sum_{z \in \V} \,Y_t (z) \,\ind \{Y_t (z) \geq 0 \} \sim M + \cl.
\end{equation}
 In other respects, our assumption that the distribution of money converges to the Laplace distribution given in~\eqref{eq:heuristics-1} implies that
\begin{equation}
\label{eq:third-equation-2}
  \begin{array}{l}
  \displaystyle \lim_{t \to \infty} E \bigg(\frac{1}{N} \sum_{z \in \V} Y_t (z) \,\ind \{Y_t (z) \geq 0 \} \bigg) \vspace*{0pt} \\ \hspace*{80pt} = \
  \displaystyle \int_{\R} x f_Y (x) \,\ind \{x \geq 0 \} \,dx =
  \displaystyle \int_{\R_+} Kxe^{- ax} \,dx = \frac{K}{a^2}. \end{array}
\end{equation}
 Combining~\eqref{eq:third-equation-1} and~\eqref{eq:third-equation-2} proves the lemma.
\end{proof} \\ \\
 Using Lemmas~\ref{lem:density-mean} and~\ref{lem:third-equation}, we can now prove~\eqref{eq:dist-collective-2} that we recall in the next lemma.
\begin{lemma} --
\label{lem:solve-system}
 Let~$\rho = \cl / M$. Then, as~$N, M/N, \cl/N \to \infty$,
 $$ K \sim \frac{1}{T} \bigg(\sqrt{1 + \rho} - \sqrt{\rho} \bigg)^2 \quad
    a \sim \frac{1}{T} \bigg(1 - \sqrt{\frac{\rho}{1 + \rho}} \bigg) \quad
    b \sim \frac{1}{T} \bigg(\sqrt{\frac{1 + \rho}{\rho}} - 1 \bigg). $$
\end{lemma}
\begin{proof}
 Lemma~\ref{lem:third-equation} and the second equation in Lemma~\ref{lem:density-mean} imply that
 $$ \frac{K}{a^2} \sim \frac{M + \cl}{N} \quad \hbox{and} \quad
    \frac{K}{b^2} = \frac{K}{a^2} - \bigg(\frac{K}{a^2} - \frac{K}{b^2} \bigg) \sim \frac{M + \cl}{N} - \frac{M}{N} = \frac{\cl}{N}. $$
 In particular, some basic algebra gives
\begin{equation}
\label{eq:solve-system-1}
  a \sim \sqrt{\frac{KN}{M + \cl}} \quad \hbox{and} \quad b \sim \sqrt{\frac{KN}{\cl}}
\end{equation}
 which, together with the first equation in Lemma~\ref{lem:density-mean}, implies that
\begin{equation}
\label{eq:solve-system-2}
  \frac{1}{K} = \frac{1}{a} + \frac{1}{b} \sim \sqrt{\frac{M + \cl}{KN}} + \sqrt{\frac{\cl}{KN}} \quad \hbox{and} \quad
           K \sim \bigg(\frac{\sqrt N}{\sqrt{M + \cl} + \sqrt{\cl}} \bigg)^2.
\end{equation}
 Combining~\eqref{eq:solve-system-1} and~\eqref{eq:solve-system-2}, we deduce that
\begin{equation}
\label{eq:solve-system-3}
  \begin{array}{rcl}
    a & \n \sim \n & \displaystyle \bigg(\frac{\sqrt N}{\sqrt{M + \cl} + \sqrt{\cl}} \bigg) \sqrt{\frac{N}{M + \cl}} = \frac{1}{\sqrt{M + \cl}} \,\bigg(\frac{N}{\sqrt{M + \cl} + \sqrt{\cl}} \bigg) \vspace*{8pt} \\
    b & \n \sim \n & \displaystyle \bigg(\frac{\sqrt N}{\sqrt{M + \cl} + \sqrt{\cl}} \bigg) \sqrt{\frac{N}{\cl}} = \frac{1}{\sqrt{\cl}} \,\bigg(\frac{N}{\sqrt{M + \cl} + \sqrt{\cl}} \bigg). \end{array}
\end{equation}
 Recalling~$T = M/N$ and~$\rho = \cl/M$, and using~\eqref{eq:solve-system-2} and~\eqref{eq:solve-system-3}, we get
 $$ \begin{array}{rcl}
      K & \n \sim \n & \displaystyle \frac{1}{T} \bigg(\frac{1}{\sqrt{1 + \rho} + \sqrt{\rho}} \bigg)^2 = \frac{1}{T} \bigg(\sqrt{1 + \rho} - \sqrt{\rho} \bigg)^2 \vspace*{8pt} \\
      a & \n \sim \n & \displaystyle \frac{1}{T} \,\sqrt{\frac{1}{1 + \rho}} \bigg(\frac{1}{\sqrt{1 + \rho} + \sqrt{\rho}} \bigg) = \frac{1}{T} \bigg(1 - \sqrt{\frac{\rho}{1 + \rho}} \bigg) \vspace*{8pt} \\
      b & \n \sim \n & \displaystyle \frac{1}{T} \,\sqrt{\frac{1}{\rho}} \bigg(\frac{1}{\sqrt{1 + \rho} + \sqrt{\rho}} \bigg) = \frac{1}{T} \bigg(\sqrt{\frac{1 + \rho}{\rho}} - 1 \bigg). \end{array} $$
 This completes the proof.
\end{proof}


\end{document}

%% file: debt.pstex_t
\begin{picture}(0,0)%
\includegraphics{debt.pstex}%
\end{picture}%
\setlength{\unitlength}{3947sp}%
\begingroup\makeatletter\ifx\SetFigFont\undefined%
\gdef\SetFigFont#1#2#3#4#5{%
  \reset@font\fontsize{#1}{#2pt}%
  \fontfamily{#3}\fontseries{#4}\fontshape{#5}%
  \selectfont}%
\fi\endgroup%
\begin{picture}(15337,15171)(-7814,80)
\put(-7199,239){\makebox(0,0)[b]{\smash{{\SetFigFont{34}{40.8}{\familydefault}{\mddefault}{\updefault}{\color[rgb]{0,0,0}$-2400$}%
}}}}
\put(-4799,239){\makebox(0,0)[b]{\smash{{\SetFigFont{34}{40.8}{\familydefault}{\mddefault}{\updefault}{\color[rgb]{0,0,0}$-1600$}%
}}}}
\put(-2399,239){\makebox(0,0)[b]{\smash{{\SetFigFont{34}{40.8}{\familydefault}{\mddefault}{\updefault}{\color[rgb]{0,0,0}$-800$}%
}}}}
\put(  1,239){\makebox(0,0)[b]{\smash{{\SetFigFont{34}{40.8}{\familydefault}{\mddefault}{\updefault}{\color[rgb]{0,0,0}$0$}%
}}}}
\put(2401,239){\makebox(0,0)[b]{\smash{{\SetFigFont{34}{40.8}{\familydefault}{\mddefault}{\updefault}{\color[rgb]{0,0,0}$800$}%
}}}}
\put(4801,239){\makebox(0,0)[b]{\smash{{\SetFigFont{34}{40.8}{\familydefault}{\mddefault}{\updefault}{\color[rgb]{0,0,0}$1600$}%
}}}}
\put(7201,239){\makebox(0,0)[b]{\smash{{\SetFigFont{34}{40.8}{\familydefault}{\mddefault}{\updefault}{\color[rgb]{0,0,0}$2400$}%
}}}}
\put(-7799,689){\makebox(0,0)[rb]{\smash{{\SetFigFont{34}{40.8}{\familydefault}{\mddefault}{\updefault}{\color[rgb]{0,0,0}$0.0000$}%
}}}}
\put(-7799,2189){\makebox(0,0)[rb]{\smash{{\SetFigFont{34}{40.8}{\familydefault}{\mddefault}{\updefault}{\color[rgb]{0,0,0}$0.0001$}%
}}}}
\put(-7799,3689){\makebox(0,0)[rb]{\smash{{\SetFigFont{34}{40.8}{\familydefault}{\mddefault}{\updefault}{\color[rgb]{0,0,0}$0.0002$}%
}}}}
\put(-7799,5189){\makebox(0,0)[rb]{\smash{{\SetFigFont{34}{40.8}{\familydefault}{\mddefault}{\updefault}{\color[rgb]{0,0,0}$0.0003$}%
}}}}
\put(-7799,6689){\makebox(0,0)[rb]{\smash{{\SetFigFont{34}{40.8}{\familydefault}{\mddefault}{\updefault}{\color[rgb]{0,0,0}$0.0004$}%
}}}}
\put(-7799,8189){\makebox(0,0)[rb]{\smash{{\SetFigFont{34}{40.8}{\familydefault}{\mddefault}{\updefault}{\color[rgb]{0,0,0}$0.0005$}%
}}}}
\put(-7799,9689){\makebox(0,0)[rb]{\smash{{\SetFigFont{34}{40.8}{\familydefault}{\mddefault}{\updefault}{\color[rgb]{0,0,0}$0.0006$}%
}}}}
\put(-7799,11189){\makebox(0,0)[rb]{\smash{{\SetFigFont{34}{40.8}{\familydefault}{\mddefault}{\updefault}{\color[rgb]{0,0,0}$0.0007$}%
}}}}
\put(-7799,12689){\makebox(0,0)[rb]{\smash{{\SetFigFont{34}{40.8}{\familydefault}{\mddefault}{\updefault}{\color[rgb]{0,0,0}$0.0008$}%
}}}}
\put(-7799,14189){\makebox(0,0)[rb]{\smash{{\SetFigFont{34}{40.8}{\familydefault}{\mddefault}{\updefault}{\color[rgb]{0,0,0}$0.0009$}%
}}}}
\put(-6899,14189){\makebox(0,0)[lb]{\smash{{\SetFigFont{34}{40.8}{\familydefault}{\mddefault}{\updefault}{\color[rgb]{0,0,0}$T = 500$}%
}}}}
\put(-6899,13439){\makebox(0,0)[lb]{\smash{{\SetFigFont{34}{40.8}{\familydefault}{\mddefault}{\updefault}{\color[rgb]{0,0,0}$m = 1000$}%
}}}}
\end{picture}%

%% file: bank.pstex_t
\begin{picture}(0,0)%
\includegraphics{bank.pstex}%
\end{picture}%
\setlength{\unitlength}{3947sp}%
\begingroup\makeatletter\ifx\SetFigFont\undefined%
\gdef\SetFigFont#1#2#3#4#5{%
  \reset@font\fontsize{#1}{#2pt}%
  \fontfamily{#3}\fontseries{#4}\fontshape{#5}%
  \selectfont}%
\fi\endgroup%
\begin{picture}(15391,15171)(-7814,80)
\put(-7199,239){\makebox(0,0)[b]{\smash{{\SetFigFont{34}{40.8}{\familydefault}{\mddefault}{\updefault}{\color[rgb]{0,0,0}$-2400$}%
}}}}
\put(-4799,239){\makebox(0,0)[b]{\smash{{\SetFigFont{34}{40.8}{\familydefault}{\mddefault}{\updefault}{\color[rgb]{0,0,0}$-1600$}%
}}}}
\put(-2399,239){\makebox(0,0)[b]{\smash{{\SetFigFont{34}{40.8}{\familydefault}{\mddefault}{\updefault}{\color[rgb]{0,0,0}$-800$}%
}}}}
\put(  1,239){\makebox(0,0)[b]{\smash{{\SetFigFont{34}{40.8}{\familydefault}{\mddefault}{\updefault}{\color[rgb]{0,0,0}$0$}%
}}}}
\put(2401,239){\makebox(0,0)[b]{\smash{{\SetFigFont{34}{40.8}{\familydefault}{\mddefault}{\updefault}{\color[rgb]{0,0,0}$800$}%
}}}}
\put(4801,239){\makebox(0,0)[b]{\smash{{\SetFigFont{34}{40.8}{\familydefault}{\mddefault}{\updefault}{\color[rgb]{0,0,0}$1600$}%
}}}}
\put(7201,239){\makebox(0,0)[b]{\smash{{\SetFigFont{34}{40.8}{\familydefault}{\mddefault}{\updefault}{\color[rgb]{0,0,0}$2400$}%
}}}}
\put(-7799,689){\makebox(0,0)[rb]{\smash{{\SetFigFont{34}{40.8}{\familydefault}{\mddefault}{\updefault}{\color[rgb]{0,0,0}$0.0000$}%
}}}}
\put(-7799,2189){\makebox(0,0)[rb]{\smash{{\SetFigFont{34}{40.8}{\familydefault}{\mddefault}{\updefault}{\color[rgb]{0,0,0}$0.0001$}%
}}}}
\put(-7799,3689){\makebox(0,0)[rb]{\smash{{\SetFigFont{34}{40.8}{\familydefault}{\mddefault}{\updefault}{\color[rgb]{0,0,0}$0.0002$}%
}}}}
\put(-7799,5189){\makebox(0,0)[rb]{\smash{{\SetFigFont{34}{40.8}{\familydefault}{\mddefault}{\updefault}{\color[rgb]{0,0,0}$0.0003$}%
}}}}
\put(-7799,6689){\makebox(0,0)[rb]{\smash{{\SetFigFont{34}{40.8}{\familydefault}{\mddefault}{\updefault}{\color[rgb]{0,0,0}$0.0004$}%
}}}}
\put(-7799,8189){\makebox(0,0)[rb]{\smash{{\SetFigFont{34}{40.8}{\familydefault}{\mddefault}{\updefault}{\color[rgb]{0,0,0}$0.0005$}%
}}}}
\put(-7799,9689){\makebox(0,0)[rb]{\smash{{\SetFigFont{34}{40.8}{\familydefault}{\mddefault}{\updefault}{\color[rgb]{0,0,0}$0.0006$}%
}}}}
\put(-7799,11189){\makebox(0,0)[rb]{\smash{{\SetFigFont{34}{40.8}{\familydefault}{\mddefault}{\updefault}{\color[rgb]{0,0,0}$0.0007$}%
}}}}
\put(-7799,12689){\makebox(0,0)[rb]{\smash{{\SetFigFont{34}{40.8}{\familydefault}{\mddefault}{\updefault}{\color[rgb]{0,0,0}$0.0008$}%
}}}}
\put(-7799,14189){\makebox(0,0)[rb]{\smash{{\SetFigFont{34}{40.8}{\familydefault}{\mddefault}{\updefault}{\color[rgb]{0,0,0}$0.0009$}%
}}}}
\put(-6899,14189){\makebox(0,0)[lb]{\smash{{\SetFigFont{34}{40.8}{\familydefault}{\mddefault}{\updefault}{\color[rgb]{0,0,0}$T = 500$}%
}}}}
\put(-6899,13439){\makebox(0,0)[lb]{\smash{{\SetFigFont{34}{40.8}{\familydefault}{\mddefault}{\updefault}{\color[rgb]{0,0,0}$\rho = 0.20$}%
}}}}
\end{picture}%